\documentclass[12pt]{amsart}
\usepackage[T1]{fontenc}
\usepackage{amsthm}
\usepackage{amsmath}
\usepackage{bbm}
\usepackage{bm}
\usepackage[margin=1.3in]{geometry}
\usepackage{cite}
\usepackage{amssymb}
\usepackage{xcolor}
\usepackage{enumitem}
\usepackage{mathrsfs}
\usepackage[
	colorlinks,
	linkcolor={black},
	citecolor={black},
	urlcolor={black}
]{hyperref}
\usepackage{todonotes}

\newcommand{\bbC}{{\mathbbm{C}}}

\newcommand{\bbR}{{\mathbbm{R}}}
\newcommand{\bbT}{{\mathbbm{T}}}
\newcommand{\bbZ}{{\mathbbm{Z}}}

\newcommand{\bolde}{{\bm{e}}}
\newcommand{\boldk}{{\bm{k}}}

\newcommand{\boldn}{{\bm{n}}}
\newcommand{\boldq}{{\bm{q}}}
\newcommand{\boldu}{{\bm{u}}}
\newcommand{\boldv}{{\bm{v}}}
\newcommand{\boldx}{{\bm{x}}}
\newcommand{\boldtheta}{{\bm{\theta}}}
\newcommand{\boldzero}{{\bm{0}}}

\newcommand{\calE}{{\mathcal{E}}}
\newcommand{\calG}{{\mathcal{G}}}
\newcommand{\calV}{{\mathcal{V}}}

\newcommand{\fund}{{\mathrm{f}}}

\newcommand{\sfE}{{\mathsf{E}}}
\newcommand{\sfG}{{\mathsf{G}}}
\newcommand{\sfV}{{\mathsf{V}}}

\DeclareMathOperator{\image}{im}

\newcommand{\iop}{{\mathrm{i}}}
\newcommand{\eop}{{\mathrm{e}}}

\DeclareMathOperator{\Leb}{Leb}
\DeclareMathOperator{\spectrum}{spec}

\newtheorem{theorem}{Theorem}[section]

\newtheorem{lemma}[theorem]{Lemma}

\theoremstyle{definition}

\newtheorem{assumption}[theorem]{Assumption}

\numberwithin{equation}{section}

\title[Measure of Periodic Spectra]{Measure of the Spectra of Periodic Graph Operators in the Large-Coupling Limit}

\author[J. Fillman]{Jake Fillman}

\date{}

\begin{document}

\begin{abstract}
    We derive a sharp criterion on the spectra of periodic discrete Schr\"odinger operators acting on connected periodic lattices: the measure of the spectrum goes to zero as the coupling constant goes to infinity if and only if there is no infinite connected path of degeneracies.
\end{abstract}

\maketitle

\hypersetup{
	linkcolor={black!30!blue},
	citecolor={red},
	urlcolor={black!30!green}
}

\section{Introduction}
Periodic Schr\"odinger operators on lattices have been studied extensively over the years as models of crystals and photonic materials;
we direct the reader to
 \cite{KorotyaevSaburova2014JMAA, Kuchment2016BAMS} for background and additional reading.
Considering the operators in question as quantum mechanical Hamiltonians, the spectrum corresponds with the allowable energies of the system and hence plays an important role in physics as well as in spectral theory.
It is well-known that the spectrum can be determined from the Floquet--Bloch theory and consists of finitely many nondegenerate closed intervals.

Motivated by recent works \cite{A-RDFS2025AHP,A-RFFL} that investigated such operators in the large-coupling limit for \emph{non-degenerate} potentials, the present work investigates the structure of the spectrum in the large-coupling limit for \emph{arbitrary} potential energy.
Our main result is a \emph{sharp topological} criterion: the measure of the spectrum goes to zero in the large-coupling limit if and only if one cannot travel to infinity along a connected equipotential path (i.e., a connected path on which the potential energy is constant).
Later in the work we will give an equivalent formulation in terms of the graph cohomology of the quotient graph.
Furthermore, one readily sees from the proof that the result can be formulated locally in energy: for any point \(a\) in the range of the potential energy function $Q$, there are spectral bands of the operator $\Delta + \mu Q$ ``near'' \(\mu a\) and the lengths of those bands tend to zero if and only if there is not a connected path to infinity along the set where the potential energy takes the value $a$.

Let us make this more precise.
Given a $\bbZ^d$-periodic graph\footnote{See Section~\ref{subsec:floquet} for a detailed definition.} $\calG = (\calV,\calE)$, we consider a periodic Schr\"odinger operator of the form
\begin{equation}
   H_Q:\ell^2(\calV) \to \ell^2(\calV) \qquad [H_Q \psi](\boldv) = Q(\boldv)\psi(\boldv) + \sum_{\boldu \sim \boldv} \psi(\boldu).
\end{equation}
We say that a subgraph \(\calG'  \subseteq \calG\) is a $Q$-\emph{flat path} if \(\calG'\) is connected   and \(Q\) is \emph{constant} on (the vertices of) \( \calG'\).

\begin{theorem} \label{t:perMeasGen}
Assume \(Q : \calV \to \bbR\) is a periodic potential on the $\bbZ^d$-periodic graph $\calG= (\calV, \calE)$. Then
\begin{equation} \label{eq:measureToZeroGen}
    \lim_{\mu \to \infty} \Leb \spectrum H_{\mu Q} = 0
\end{equation}
if and only if \(\calG\) does not contain any $Q$-flat paths containing infinitely many points.
\end{theorem}

Equivalently, the measure of the spectrum goes to zero in the large-coupling limit if and only if there is no point $a$ in the range of $Q$ such that $Q^{-1}\{a\}$ contains an unbounded connected component.  
This is naturally dual to works on the discrete \emph{Bethe--Sommerfeld conjecture} \cite{EmbFil2019JST, FilHan2020JdAM, HanJit2018CMP}, which considers the case in which the potential energy is small compared to the kinetic term, that is, the limit $\mu \to 0$. Estimates on the measure of the spectra of periodic operators are important and have been studied frequently over the years; in addition to the references already discussed, we mention \cite{GolKut2019GIGMA, KorotyaevSaburova2014JMAA, 
KorotyaevSaburova2015PAMS,
KorotyaevSaburova2020MathAnn, KorotyaevSaburova2022JMAA, KorotyaevSaburova2022CPAA}.
In fact, let us point out that a special case of Theorem~\ref{t:perMeasGen} can be derived from  \cite[Theorem~2.2]{KorotyaevSaburova2014JMAA}, which handles the \emph{generic} situation in which $Q$ is injective on the quotient graph.
Thus, the main achievement of the present work is to handle the case of arbitrary degenerate potentials.

Discrete Schr\"odinger operators on $\ell^2(\bbZ^d)$ have been studied quite extensively over the years: we direct the reader to the textbooks \cite{AizenmanWarzel2015GSM, CFKS1987, DF2022ESO1, LukicBook} for background information and further reading.
In recent years, there has been significant interest in the setting of periodic operators, as these are an exciting test case for photonic materials and metamaterials \cite{Kuchment2001, Joannopoulos2008}.
Such operators exhibit many interesting features, including flat bands \cite{FaustKachkovskiyPreprint, SabriYoussef2023JMP},
dispersive transport \cite{DFY2025disperse, MiZhao2020JMAA, MiZhao2022PAMS, SKW2026JMAA},
ballistic motion \cite{AVWW2011JMP, AschKnauf1998Nonlin, DLY2015CMP, Fillman2021OTAA, BoutetSabri2023OTAA},
and suppression of transport \cite{A-RCSW, A-RFFL,A-RDFS2025AHP}.
Recently, interdisciplinary approaches combining spectral theory with techniques from algebra and combinatorics have led to further results  \cite{
FLM2024JFA, 
FisLiShi2021CMP, LiShipman2020LMP, Liu2022GAFA, Shipman2020JST, ShipmanSottilepreprint}.
\smallskip

We review a few preliminaries regarding Floquet theory and topology of graphs in Section~\ref{sec:prelim} and then give the proof of Theorem~\ref{t:perMeasGen} in Section~\ref{sec:proofs}.

\subsection*{Acknowledgements}
The author is grateful to Mark Embree for helpful discussions and insights and to the Banff International Research Station for hospitality during a recent workshop, at which portions of this work were initiated and completed.
The author was supported in part by National Science Foundation Grant DMS 2513006 and Simons Foundation Grant MPS-TSM 00013720.

\section{Preliminaries} \label{sec:prelim}

\subsection{Floquet Theory} \label{subsec:floquet}
Let us review a few basic notions from Floquet theory; this is standard and can be found in a few places, e.g., \cite{FLM2024JFA, KorotyaevSaburova2014JMAA, ShipmanSottilepreprint}.

As usual, a \emph{graph} \(\calG = (\calV, \calE)\) comes with two pieces of data, a set \(\calV\) of \emph{vertices} and a set \(\calE \subseteq \calV \times \calV\) of \emph{edges}, which are \emph{ordered} pairs of vertices.
We write \(\boldu \sim \boldv\) if \((\boldu, \boldv) \in \calE\), and assume that $\calE$ is symmetric in the sense that $\boldu \sim \boldv$ if and only if $\boldv \sim \boldu$.
We also assume that there are no self-loops, that is, \( (\boldu , \boldu) \notin \calE\)  for every \(\boldu\).

A $\bbZ^d$-periodic graph is a locally finite graph $\calG = (\calE,\calV)$ equipped with a suitable action by $\bbZ^d$; it is convenient to view the vertices of the graph as points in $\bbR^d$, so we write the action of $\bbZ^d$ additively and identify the edge $(\boldu,\boldv)$ with the line segment connecting $\boldu$ to $\boldv$.
We assume that the translation action is
\begin{itemize}
    \item \emph{free}: $\boldu + \boldn \neq \boldu$ for all $\boldu \in \calV$ and $\boldn \neq \boldzero$  
    \item \emph{cofinite:} there are finitely many orbit classes of vertices and edges
\end{itemize}
Notice that if $\Gamma$ is any full-rank subgroup of $\bbZ^d$ with linearly independent generators $\{\bm{\gamma}_1,\ldots,\bm{\gamma}_d\}$, then  one can view $\calG = (\calV, \calE)$ as a $\bbZ^d$-periodic graph with respect to the rescaled $\bbZ^d$ action given by
\begin{equation}
    \boldu \dot{+}\bm{n} = \boldu + n_1\bm{\gamma}_1 + \cdots + n_d \bm{\gamma}_d.
\end{equation} In particular, by passing to a suitable subgroup, we can (and do) assume without loss of generality the following:
\begin{assumption}\mbox{\,}\label{assumption:main}
\begin{itemize}
    \item For each $\boldu \in \calV$, there is no $\boldn \in \bbZ^d$ such that $\boldu \sim  \boldu + \boldn$.
    \item For each $\boldu, \boldv \in \calV$, there is at most one $\boldn \in \bbZ^d$ such that $\boldu \sim \boldv + \boldn$.
\end{itemize}
These assumptions are not essential, but they simplify a few aspects of the analysis, as they ensure that the quotient graph $\calG/\bbZ^d$ has no self-loops and no pairs of vertices with multiple edges.
\end{assumption}

A potential \(Q : \calV \to \bbR\) is said to be \emph{periodic} if there is a full-rank lattice $\Gamma \subseteq \bbZ^d$ such that $Q(\boldu+\bm{\gamma})\equiv Q(\boldu)$ for all $\bm{\gamma} \in \Gamma$;  rescaling the group action as before, we are free to assume
\begin{equation} \label{eq:pjPerDef}
    Q(\boldu + \boldn) = Q(\boldu) \quad \text{for all } \boldn \in \bbZ^d, \ \boldu \in \calV.
\end{equation}

In this setting, there is a canonical direct integral decomposition of \(H\) that can be used to study the spectrum and spectral properties of \(H\).
Here we mostly follow the notation and conventions of \cite{KorotyaevSaburova2014JMAA}, to which the reader may refer for proofs and further discussion.

Let us describe the quotient graph $ \calG/\bbZ^d =: \sfG_\fund =(\sfV_\fund, \sfE_\fund) $ in detail. Put
\[\sfV_\fund = [0,1)^d \cap \calV \cong \calV / \Gamma,\]
and let $\nu = \#\sfV_\fund$ denote the number of vertices in the fundamental cell.
Thus, every vertex $\boldu \in \calV$ can uniquely be written as $\boldu = \boldv + \boldn $ for some $\boldv \in \sfV_\fund$ and $\boldn \in \bbZ^d$; define
\begin{equation}
    [\boldu]=\boldn, \quad \widetilde{\boldu} = \boldv.
\end{equation}
For an edge $(\boldu, \boldv) \in \calE$, we write $\tau(\boldu, \boldv) = [\boldv] - [\boldu]$ for the \emph{edge index}.
In the quotient graph, two vertices $\boldu , \boldv \in \sfV_\fund$ are joined by an edge in $\sfE_\fund$ if and only if $\boldu \sim \boldv + \boldn$ for some $\boldn$ (which is unique by Assumption~\ref{assumption:main}); we define $\sigma(\boldu,\boldv) = \tau(\boldu, \boldv+\boldn) = \boldn$ in this case.

For \( \boldtheta \in \bbT^d = \bbR^d/ \bbZ^d\), the Floquet matrix \(H(\boldtheta) \in \bbC^{\sfV_\fund \times \sfV_\fund}\) is given by  
\begin{equation}
    \langle \bolde_\boldu, H(\boldtheta) \bolde_\boldv \rangle
    = \begin{cases}
        Q(\boldu) & \boldu = \boldv \\
        \eop^{2\pi \iop \langle \sigma(\boldu, \boldv), \boldtheta \rangle} & \boldu \sim \boldv \\
        0 & \text{otherwise,}
    \end{cases}
\end{equation}
where as usual $\{\bolde_\boldv : \boldv \in \sfV_\fund\}$ denotes the standard basis of $\bbC^{\sfV_\fund}$. 
Equivalently, $H(\boldtheta)$ can be interpreted as the restriction of \(H\) to \(\sfV_\fund\) with the boundary conditions
\begin{equation}
\psi(\boldu+ \boldn) = \eop^{2\pi \iop \langle \boldn, \boldtheta \rangle} \psi(\boldu),
\quad \boldn  \in \bbZ^d.
\end{equation}

Note that $\tau$ and $\sigma$ are antisymmetric, so $H(\boldtheta)$ is self-adjoint and therefore has \( \nu \) real eigenvalues, denoted with multiplicity by
\begin{equation}
    \lambda_1(\boldtheta) \leq \cdots \leq \lambda_\nu(\boldtheta).
\end{equation}
The \(j\)th band of \(H\) is the range of the function \(\lambda_j(\cdot)\) and the spectrum of \(H\) is the union of the bands:
\begin{equation}
I_j:= \{\lambda_j(\boldtheta) : \boldtheta \in \bbT^d\}, \quad     \spectrum H = \bigcup_{j=1}^\nu I_j.
\end{equation}

We furthermore view \(\sfG_\fund\) (its vertices and edges) as a topological object: a subset of the torus
\begin{equation}
    \bbT^d := \bbR^d /  \bbZ^d.
\end{equation}
In particular, if $\boldu\sim \boldv$ in $\sfG_\fund$, it means that $\boldu \sim \boldv + \sigma(\boldu,\boldv)$ in $\calG$, so we view the edge $(\boldu,\boldv)$ in $\sfE_\fund$ as the projection of the line segment $(\boldu, \boldv + \sigma(\boldu,\boldv)$ to $\bbT^d$.

Here and throughout the paper, we use a few standard notions from algebraic topology; the reader may consult any standard reference (e.g.\ Hatcher \cite{Hatcher2002algebraicTopology}) for the necessary background.
By elementary algebraic topology, it is known that the fundamental group of \(\bbT^d\) is
\begin{equation}
    \pi_1(\bbT^d) \cong \bbZ^d
\end{equation}
with the canonical projection \(\Pi: \bbR^d \to \bbT^d\) furnishing the universal cover.
For our purposes, it is convenient to realize this equivalence in the following manner.
Fix a base point \({\boldx}_0 \in \bbT^d\). 
Given a continuous path \(\gamma:[0,1] \to \bbT^d\) with \(\gamma(0) = \gamma(1) = {\boldx}_0\), \(\gamma\) has a lift \(\widetilde\gamma:[0,1] \to \bbR^d\) satisfying \(\Pi \circ \widetilde \gamma = \gamma\).
The map sending \(\gamma\) to \(\widetilde\gamma(1) - \widetilde \gamma(0)\) is well-defined modulo homotopy of loops and furnishes an isomorphism \(\pi_1(\bbT^d) \to \bbZ^d\).

\subsection{Graph Cohomology} \label{subsec:graphCohomology}

It will be useful to formulate a few topological and graph-theoretic notions to facilitate the discussion.

Let us briefly reprise the standard constructions of cohomology groups of a \emph{finite} graph, $ \sfG = ( \sfV,  \sfE)$ (we use a different font to distinguish these from the infinite graphs considered elsewhere in the manuscript).
The main graph we have in mind is a subgraph of \(\sfG_\fund\) on which the potential \(Q\) is constant, so among other things this graph is not necessarily connected. As before, we assume our graphs are symmetric: $\boldu \sim \boldv \iff \boldv \sim \boldu$.

Let  \(\Omega^0(\sfG) = \bbR^\sfV \) and let \(\Omega^1(\sfG) \subseteq \bbR^\sfE\) be the set of \emph{antisymmetric} functions, that is, those $\psi:\sfE \to \bbR$ such that $\psi(\boldu, \boldv) = - \psi(\boldv, \boldu)$ for all $(\boldu , \boldv)\in \sfE$.
Evidently, $\dim \Omega^1$ is $\#\sfE/2$, the number of ``unoriented'' edges in $\sfG$.
Define the operator \(\nabla: \Omega^0(\sfG) \to \Omega^1(\sfG)\)  by
\begin{equation}
    [\nabla \varphi](\boldu, \boldv) = \varphi(\boldv) - \varphi(\boldu).
\end{equation}
It is not hard to check that \(H^0(\sfG) := \ker \nabla \cong \bbR^{\beta_0}\), where \( \beta_0 \) is the number of connected components of \(\sfG\).
By a dimension-counting argument, one has \(H^1(\sfG) := \Omega^1(\sfG)/\image \nabla \cong \bbR^{\beta_1}\), where 
\[\beta_1 = \dim \Omega^1(\sfG)  - \dim\Omega^0(\sfG) + \beta_0 \]
 is the first Betti number of the graph, which amounts to the number of ``independent loops'' in the graph.
 In particular \(\beta_1 = 0\) if and only if \(\sfG\) is a union of trees.

The following lemma is well-known; we include the proof for the reader's convenience.
Here, we define a \emph{vertex loop} to be a finite sequence $(\boldv_0,\ldots, \boldv_\ell)$ of points of $\sfV$ such that $\boldv_0 = \boldv_\ell$ and $\boldv_{j-1}\sim \boldv_j$ for all $j$.
\begin{lemma}\label{lem:closedLoopH1}  
Given $\psi\in \Omega^1(\sfG)$, the following are equivalent:
\begin{enumerate}
\item[\rm{(a)}] The cohomology class of $\psi$ is trivial in $H^1(\sfG)$.
\item[{\rm(b)}] For any   vertex loop  $(\boldv_0, \boldv_1,\ldots, \boldv_\ell)$ in $\sfG$, one has
\begin{equation} \label{eq:clsedLoopSum}
    \sum_{j=1}^\ell  \psi(\boldv_{j-1} , \boldv_j) = 0.
\end{equation}
\end{enumerate} 
\end{lemma}

\begin{proof}
If $\psi = \nabla\varphi$, then \eqref{eq:clsedLoopSum} holds for every closed loop by a computation.
\medskip

Conversely, assume \eqref{eq:clsedLoopSum} holds for all closed loops. Choose one vertex from each connected component, call those vertices $\boldv_1,\ldots,\boldv_{\beta_0}$, and define $\varphi(\boldv_j) = 0$ for each $j$.
Thereafter define $\varphi$ inductively so that $\psi = \nabla \varphi$ holds.
The condition \eqref{eq:clsedLoopSum} holding for all closed loops ensures that $\varphi$ is well-defined.
\end{proof}

Any \(\sfV_0 \subseteq \sfV\), can be given the structure of a graph by including all edges in \(\sfE\) for which both vertices lie in \(\sfV_0\), and we write \(\sfG_0 = (\sfV_0, \sfE_0)\) for the induced subgraph.\footnote{Strictly speaking, $\sfE_0$ and $\sfG_0$ depend on the choice of $\sfV_0$ and hence properly should be written as  $\sfE_0(\sfV_0)$ and $\sfG_0(\sfV_0)$, but we suppress the dependence in the notation.}
There are then natural projections 
\[
\Omega^k(\sfG) \to \Omega^k(\sfG_0): \psi \mapsto \psi^{\sfV_0}, \quad k =0,1,
\]
which are simply given by restricting to \(\sfG_0\).

\section{Proof of Main Result}
\label{sec:proofs}

Let us now fix a $\bbZ^d$-periodic graph $\calG = (\calV, \calE)$ and periodic potential $Q:\calV \to \bbR$ and consider (subgraphs of)
\begin{equation}
    \sfG_\fund = (\sfV_\fund,\sfE_\fund) := \calG/\bbZ^d
\end{equation}
in the framework of Section~\ref{subsec:graphCohomology}.

The Floquet matrix \(\Delta(\boldtheta)\) of the free Laplacian on $\calG$ induces a 1-chain \(\psi_\boldtheta \in \Omega^1(\sfG_\fund)\)  by 
\begin{equation}
 \psi_\boldtheta(\boldu, \boldv) =
 \langle \sigma(\boldu, \boldv),\boldtheta \rangle
 \end{equation}
so that in particular:
\begin{equation}
    \langle \bolde_\boldu, \Delta(\boldtheta) \bolde_\boldv \rangle
    = \eop^{2\pi \iop  \psi_\theta(\boldu, \boldv)}
\end{equation}
for $\boldu \sim \boldv$.

The main topological input is that for a nonempty subset \(\sfV_0 \subseteq \sfV_\fund\), the chains obtained by restricting \(\psi_\boldtheta\) are nontrivial in \(H^1(\sfG_0)\) if and only \(\sfG_0\) (as a subset of \(\bbT^d \))  contains a closed vertex loop that is  nontrivial in the homotopy group \(\pi_1(\bbT^d)\). 
This is a slight abuse of terminology: a vertex loop is by definition a finite sequence of vertices $(\boldv_0,\ldots,\boldv_\ell)$ with $\boldv_{j-1} \sim \boldv_j$ for each $j$ and $\boldv_0=\boldv_\ell$; the corresponding element of the homotopy group is the obvious one.

\begin{lemma} \label{lem:topologicalChar}
Consider a nonempty set \(\sfV_0\subseteq \sfV_\fund\) and its associated graph \(\sfG_0 = (\sfV_0, \sfE_0)\).
    The following are equivalent.
\begin{enumerate}[label={\rm(\alph*)}]
    \item 
    \label{item:nontrivInGraph}
    \(\psi_{\boldtheta}^{\sfV_0}\) is nontrivial in \(H^1(\sfG_0)\) for some \(\boldtheta\).
    
    \item 
    \label{item:infinitePath}
    The lift of \(\sfG_0\) from \(\bbT^d\) to \(\bbR^d\) {\rm(}i.e., \(\Pi^{-1}(\sfG_0)\){\rm)} contains a vertex path having infinitely many distinct vertices.

    \item   
    \label{item:nontrivInTorus}
    \( \sfG_0 \) contains a vertex loop that is nontrivial in \( \pi_1(\bbT^d)\).
\end{enumerate}\end{lemma}

\begin{proof}
\underline{\ref{item:nontrivInGraph}\(\implies\)\ref{item:infinitePath}}.
Assume \(\psi_\boldtheta^{\sfV_0}\) is nontrivial in \(H^1(\sfG_0)\).
According to Lemma~\ref{lem:closedLoopH1} there is a closed loop \((\boldu_0,\ldots, \boldu_\ell)\) in \(\sfV_0\)
such that \(\boldu_0 = \boldu_\ell\), and
\begin{equation}
    \sum_{r=1}^{\ell} \psi_\boldtheta(\boldu_{r-1}, \boldu_r)  \neq 0.
\end{equation}
Denote
\begin{equation} \label{eq:torusLoopIncrements}
\bm{k}
:= 
\sum_{r=1}^{\ell} \sigma(\boldu_{r-1}, \boldu_r) ,
\end{equation}
which immediately gives us
\begin{equation} \label{eq:psiTildeLoopSum}
    \sum_{r=1}^{\ell} \psi_\boldtheta(\boldu_{r-1}, \boldu_r)  
    = \langle \boldk, \boldtheta\rangle.
\end{equation}
In particular $\langle \boldk, \boldtheta \rangle \neq 0$, whence $\boldk\neq \boldzero$.

Without loss, we further assume that the loop is irreducible in the sense that \( \boldu_r \neq  \boldu_0\) for all \(r \neq 0,\ell\).
Lifting this loop through \(\Pi\) to \(\bbR^d\) produces a vertex path from \(\boldu_0\)  to \(\boldu_0 + \boldk\) through the preimage of \(\sfG_0\).
By periodicity, we produce a vertex path from \(\boldu_0\) to \(\boldu_0+s\boldk\) for every \(s \in \bbZ\), establishing \ref{item:infinitePath}.
\medskip

\underline{\ref{item:infinitePath}\(\implies\)\ref{item:nontrivInTorus}}.
If the lift of \(\sfG_0\) to \(\bbR^d\) contains an infinite connected vertex path, there in particular exists a connected vertex path from some \(\boldu\) to \(\boldu+\boldk\) with \(\boldk \in \bbZ^d \setminus\{{0}\}\).
The projection of this path to \(\bbT^d\) produces a vertex path that is nontrivial in \(\pi_1(\bbT^d)\) on account of the  correspondence principle for covering spaces.
\medskip

\underline{\ref{item:nontrivInTorus}\(\implies\)\ref{item:nontrivInGraph}}.
Suppose \(\sfG_0\)  contains a homotopically nontrivial vertex loop \(\gamma\), let \(\widetilde \gamma\) denote a lift of \(\gamma\) through \(\Pi\), and put
 \(\boldk = \widetilde\gamma(1) - \widetilde\gamma(0)\). 
 It follows that there is a closed loop in \(\sfG_0\) satisfying \eqref{eq:torusLoopIncrements}, so the desired \(\boldtheta\) is then any for which \(\langle \boldk, \boldtheta \rangle \neq 0\), which can be seen directly from \eqref{eq:psiTildeLoopSum}.
\end{proof}

We now have what we need to prove the main result.

\begin{proof}[Proof of Theorem~\ref{t:perMeasGen}]
Let \(Q:\calV \to \bbR\) be periodic. As discussed before, we can replace the $\bbZ^d$ action by the action by a full-rank subgroup and ensure that \eqref{eq:pjPerDef} and Assumption~\ref{assumption:main} hold.
\medskip

\noindent 
\textbf{Case 1: \boldmath \(Q\) has no infinite flat paths.} 
Consider a value \(a\) in the range of \(Q\),  let
    \(\sfV_\fund^a \subseteq \sfV_\fund\) be given by
    \begin{equation}
        \sfV_\fund^a = \{\boldu \in \sfV_\fund : Q(\boldu) = a\},
    \end{equation}
    and write $\sfG^a_\fund$ for the corresponding graph.
    Since \(Q\) has no infinite chains, Lemma~\ref{lem:topologicalChar} implies that \(\psi_\boldtheta^{\sfV_\fund^a}\) is trivial in \(H^1(\sfG_\fund^a)\) for all \(\boldtheta\), allowing us to write
    \begin{equation}
        \psi_\boldtheta^{\sfV_\fund^a} = \nabla \varphi
    \end{equation}
for some $\varphi = \varphi_\boldtheta^{\sfV_\fund^a}$ in $\Omega^0(\sfG_\fund^a)$.
    Considering the diagonal unitary matrix $U$ given on $\Omega^0(\sfG_\fund^a)$ by $[U\eta](\boldu) = \eop^{-2\pi \iop \varphi(\boldu)} \eta(\boldu)$, we have
    \begin{align}
    \nonumber
        \langle \bolde_\boldu, U^*\Delta(\boldtheta) U \bolde_\boldv \rangle
=         \eop^{2\pi \iop (\varphi(\boldu) - \varphi(\boldv))}\langle \bolde_\boldu, \Delta(\boldtheta)  \bolde_\boldv \rangle
&=         \eop^{-2\pi \iop \psi_\boldtheta(u,v)} \eop^{2\pi \iop \psi_\boldtheta(\boldu, \boldv)}\langle \bolde_\boldu, \Delta(\boldzero)  \bolde_\boldv \rangle \\
& = \langle \bolde_\boldu, \Delta(\boldzero)  \bolde_\boldv \rangle .
    \end{align}
    Thus, the restriction \( \Delta(\boldtheta)|_{\sfV_\fund^a}\) is unitarily equivalent to \(\Delta(\boldzero)|_{\sfV_\fund^a}\) for each $\boldtheta$.
    
    By eigenvalue perturbation theory (e.g., compare \cite[Section~II.2.3]{Kato} or \cite{MorBurOve1997SIMAX}) for small \(\varepsilon\), the eigenvalues of \(A(\varepsilon,\boldtheta) := Q + \varepsilon \Delta(\boldtheta) \) near \(a\) enjoy the following asymptotic expansions
    \begin{equation}
        a+ c_{a,r} \varepsilon + O(\varepsilon^2),
    \end{equation}
    where \(\{c_{a,r} : r=1,\ldots, \#\sfV_\fund^a\}\) are the (common) eigenvalues of all \(\Delta(\boldtheta)|_{\sfV_\fund^a}\).
    It follows that  the bands of \(A(\varepsilon) := Q + \varepsilon \Delta\) near \(a\) have length \(\lesssim \varepsilon^2\), so the bands of \(H_{\mu Q}\) near \(\mu a\) have length \(\lesssim \mu^{-1}\), concluding the argument in this case.
    \bigskip

\noindent    \textbf{Case 2: \boldmath \(Q\) has an infinite flat path.} Suppose \(\{\boldu \in \calV : Q(\boldu) =a\}\) contains an infinite flat path, and let us again consider the bands of \(A(\varepsilon)\) near \(a\).
    As before, the eigenvalues of \(A(\varepsilon,\boldtheta)\) near \(a\) enjoy asymptotic expansions of the form
    \[
    a + c_{a,r}(\boldtheta) \varepsilon + O(\varepsilon^2),
    \]
    where \(c_{a,r}(\boldtheta)\) denote the eigenvalues of \(\Delta(\boldtheta)|_{\sfV_\fund^a}\).
    According to standard results for connected periodic graphs (e.g., \cite{KorotyaevSaburova2014JMAA, SabriYoussef2023JMP}) it cannot happen that all bands of \(\Delta|_{Q^{-1}\{a\}}\) are flat.
    Thus, up to some \(O(\varepsilon^2)\) errors, the bands of \(A(\varepsilon)\) near \(a\) are supplied by  $a + \varepsilon J_i$ where $\{J_i\}$ denote the bands of \(\Delta|_{Q^{-1}\{a\}}\); since not all bands of \(\Delta|_{Q^{-1}\{a\}}\) are flat, at least one of the bands of $A(\varepsilon)$ near the point $a$ has length of order \(\varepsilon\).
    It follows that the spectrum of \(H_{\mu Q}\) has length of order one, concluding the proof.
\end{proof}

It is clear from the discussion and proof that the asymptotic lower bound on the measure of the spectrum of $H_{\mu Q}$ is dictated by the measure of the spectrum of the free operator restricted to subgraphs  on which $Q$ is constant.
In special cases, one can sometimes derive explicit lower bounds. For instance, in the case $\calG = \bbZ^d$ (with nearest-neighbor connections between vertices) if \(Q\) is \(\boldq\) periodic (i.e.\ $Q(\boldu + q_j \bolde_j) \equiv Q(\boldu)$) with  \(\min q_j=1\), then the associated operator is \emph{separable} and this allows one to prove quantitative lower bounds.
Without loss, assume \(q_d = q_{d-1} = \cdots = q_{d-m+1}=1\). 
It follows that there is some \(W:\bbZ^{d-m} \to \bbR\) that is \((q_1,...,q_{d-m})\)-periodic and that satisfies
\begin{equation}
Q(\boldu) = W(u_1,\ldots,u_{d-m}), \quad \boldu = (u_1,\ldots,u_d) \in \bbZ^d.
\end{equation}
Write 
\[
[\Delta_j \psi](\boldu) = \psi(\boldu-\bolde_j) + \psi(\boldu+\bolde_j), \quad
\Delta_{(i,j)} = \Delta_i + \Delta_{i+1} + \cdots + \Delta_j.\]
Viewing 
\[ \ell^2(\bbZ^d) \cong \ell^2(\bbZ^{d-m}) \otimes \ell^2(\bbZ^m),\]
we can write
\begin{equation}
H_{\mu Q} \cong \left(\Delta_{(1,d-m)}+ \mu W\right) \otimes I + I \otimes \Delta_{(d-m+1,d)}
\end{equation}
from which it follows that
\begin{align}
\spectrum H_{\mu Q}
& = \spectrum(\Delta_{(1,d-m)}+\mu W) + \spectrum \Delta_{(d-m+1,d)} \\
& =\spectrum(\Delta_{(1,d-m)}+\mu W) + [-2m,2m].
\end{align}
Consequently,
\begin{equation} \label{eq:measureNotToZero}
\Leb\spectrum H_{\mu Q} \geq 4mr,
\end{equation}
where \(r = \#\{ Q(n) : n \in \bbZ^d\}\).
\bigskip

\bibliographystyle{abbrv}
\bibliography{refs}

\end{document}